\documentclass[11pt,reqno]{amsart}
\usepackage{amsthm, amsfonts, amssymb, color}
 \usepackage{mathrsfs}
\usepackage{enumerate}
\usepackage{amsmath}
 \usepackage{amstext, amsxtra}
  \usepackage{txfonts}
 \usepackage[colorlinks, linkcolor=black, citecolor=blue, pagebackref, hypertexnames=false]{hyperref}
\usepackage{graphicx}
\usepackage[symbol]{footmisc}
 \allowdisplaybreaks
\newtheorem{theorem}{Theorem}[section]
\newtheorem{proposition}[theorem]{Proposition}
\newtheorem{corollary}[theorem]{Corollary}
\newtheorem{lemma}[theorem]{Lemma}

\newtheorem{example}[theorem]{Example}
\newtheorem{remark}[theorem]{Remark}

\renewcommand\Re{\operatorname{Re}}

\numberwithin{equation}{section}
\theoremstyle{definition}

\title
[Inhomogeneous Oscillatory Integrals and Global Smoothing Effects]
 {Inhomogeneous Oscillatory Integrals and Global Smoothing Effects for Dispersive Equations}
\author{Tianxiao Huang, Shanlin Huang and Quan Zheng}
\address{Tianxiao Huang, School of Mathematics and Statistics, Huazhong University of Science and Technology, Wuhan, Hubei, 430074, P.R.China}
\email{htx@hust.edu.cn}
\address{Shanlin Huang\footnote{Corresponding author.}, School of Mathematics and Statistics, Huazhong University of Science and Technology, Wuhan, Hubei, 430074, P.R.China}
\email{shanlin\_huang@hust.edu.cn}
\address{Quan Zheng, School of Mathematics and Statistics, Huazhong University of Science and Technology, Wuhan, Hubei, 430074, P.R.China}
\email{qzheng@hust.edu.cn}
\date{\today}

\begin{document}

\maketitle
\begin{abstract}
We study oscillatory integrals of the type ${\mathcal F}^{-1}(e^{ita(\cdot)}\psi(\cdot))$ where $a$ is a general function satisfying some elliptic type and non-degenerate conditions at both the origin and the infinity, and $\psi$ belongs to some symbol class. Point-wise estimates in space-time are gained with partial sharpness. As applications, global smoothing effects of $L^p-L^q$ as well as Strichartz type for dispersive equations are studied. An application to fractional Schr\"odinger equations is also given.
\end{abstract}


\section{Introduction}
\setcounter{equation}{0}
This paper is a following study of Kenig, Ponce and Vega \cite{KPG}. There, the local and global smoothing effects for a class of dispersive Cauchy problems
\begin{equation}\label{eq01}
\begin{split}
\begin{cases}\partial_tu-ia(D)u=0,\quad x\in\mathbb{R}^n,~t\in\mathbb{R},\\
u(x,0)=u_0(x),
\end{cases}
\end{split}
\end{equation}
were studied, where $D=-i\nabla_x$ and $a(D)$ is defined through its Fourier symbol $a(\xi)$. Focusing on the global smoothing effects, it was shown, for example in dimension one, that if $a(\xi)$ is a real polynomial of degree $m\geq2$,
\begin{equation}\label{eq02}
W_\gamma(t)u_0(x):=\int_\mathbb{R}e^{i(ta(\xi)+x\xi)}|a''(\xi)|^\frac\gamma2\hat{u}_0(\xi)d\xi,\quad (x,t)\in\mathbb{R}^2,
\end{equation}
where $\gamma\geq0$ and $\hat{u}_0$ denotes the Fourier transform of $u_0$, then for any $\theta\in[0,1]$ we have
\begin{equation}\label{eq03}
||W_{\theta/2}(t)u_0||_{L_t^q(\mathbb{R};~L_x^p)}\leq C||u_0||_{L_x^2}.
\end{equation}
By interpolation, the key ingredient to prove (\ref{eq03}) is the following dispersive estimate:
\begin{equation}\label{eq04}
||W_1(t)u_0||_{L_x^\infty}\leq C|t|^{-\frac12}||u_0||_{L_x^1},
\end{equation}
which roughly means that the solution $u$ to (\ref{eq01}) has $\frac{m-2}{2}$ derivatives in $L_x^\infty$ if $u_0\in L_x^1$. The proof is to show that the "$|a''|^{1/2}$-derivative" of the convolution kernel of $W_1(t)$, appearing as an oscillatory integral, has the decay $|t|^{-1/2}$ uniformly in the space variable. This result was also generalized for a class of phase functions $a$ in \cite{KPG}, but not in the higher dimensional version, which would be much harder and was also established there, while $a$ is assumed to be a real polynomial of degree $m\geq2$ having (non-degenerate) elliptic principle part, $a''$ in (\ref{eq02}) is replaced by the Hessian $Ha$, and the integrand is cut off away from the origin. A special case that $a$ is non-elliptic was also considered. We refer to \cite{KPG} for the motivations and more thorough statements, as well as the local smoothing effect which is out of the scope of the present paper.

The main ingredient of this paper is to show, in a very general setting of $a(\xi)$ more than the type of polynomial, that even in higher dimensions, there are actually much more global smoothing effects for (\ref{eq01}) than those associated with $Ha$, by studying the oscillatory integral
\begin{equation}\label{eq05}
I(t,x)=\int_{\mathbb{R}^n} e^{i(ta(\xi)+x\cdot\xi)}\psi(\xi)d\xi,\quad t\in\mathbb{R}\setminus\{0\},~x\in\mathbb{R}^n,
\end{equation}
where $\psi$ stands for the smoothing. One of the features in this study claims that, if the growth of $\psi$ differs from that of $Ha$, then $I$ may have decays both in $|t|$ and in $|x|$. This leads to more possible types of estimates for the solution $u$ to (\ref{eq01}), typically the $L^p-L^q$ type estimates. The other important feature is the generality of $a$ in our main result Theorem \ref{th1}. To give a first sight, let's recall an interesting result due to Ben-Artzi, Koch and Saut \cite{AKS}:

\begin{theorem}\label{th0}
If $a(\xi)=|\xi|^4+|\xi|^2$, $\psi(\xi)=\xi^\alpha$ for any multi-index $\alpha$, then
\begin{equation*}
\begin{split}
|I(t,x)|\leq\begin{cases}C|t|^{-\frac{n+|\alpha|}{4}}(1+|t|^{-\frac14}|x|)^{-\frac{n-|\alpha|}{3}}~&\text{if}~0<|t|\leq1~\text{or}~|x|\geq|t|,\\
C|t|^{-\frac{n+|\alpha|}{2}}(1+|t|^{-\frac12}|x|)^{|\alpha|}~&\text{if}~|t|>1~\text{and}~|x|<|t|.
\end{cases}
\end{split}
\end{equation*}
\end{theorem}
We shall recover this result in Example \ref{ex1}. Since the phase function $a$ is radial, in the polar coordinates, the proof in \cite{AKS} uses estimates of the Bessel functions for the sphere part, and the stationary phase argument for the radial part. What contribute in the radial part are the quantities of $\partial_sa$, $\partial_s^2a$ and $\psi$ where $s=|\xi|$, both near and away from the origin. This suggests us to study (\ref{eq05}) in a compact domain and its complement, which indicates the class of phase $a$ to have different characteristics at some finite point and at the infinity, besides the variety of $\psi$. For example, $a$ behaves like $|\xi|^{m_1}$ at $0$ and behaves like $|\xi|^{m_2}$ at infinity, where $m_1,~m_2>1$ can be real numbers.

The authors don't know if there exist general studies covering the above result. For $\psi$ supported away from the origin, A. Miyachi studied general radial phases in \cite{Mi}, typical homogeneous phases in \cite{Mi3}, and then general homogeneous phases of degree $1$ thoroughly in \cite{Mi2} which, under its spirit, seems to imply results for general homogeneous phases of other degrees. For $\psi$ supported on $\mathbb{R}^n$, Zheng, Yao and Fan \cite{ZYF} studied homogeneous elliptic polynomial phase $a$, where $\psi=1$ and the level surface $\Sigma_a=\{\xi\in\mathbb{R}^n;~a(\xi)=1\}$ is convex of finite type. Cui \cite{C1} studied inhomogeneous elliptic polynomial phase $a$, assuming the principal part $a_m$ has non-degenerate $Ha_m$ everywhere in $\mathbb{R}^n\setminus\{0\}$, (which is equivalent to the non-vanishing Gaussian curvature of $\Sigma_{a_m}$ everywhere,) and gained short time estimates by a scaling approach, while $\psi$ is allowed to be in a symbol class. For long time results, Balabane and Emami-Rad \cite{BE} was an early attempt, gaining rough estimates but in different perspectives compared with Cui \cite{C1}. Kim, Arnold and Yao \cite{KAY} improved the long time results. Ding and Yao \cite{DY} also considered the case where $Ha$ is positive definite, under some assumptions due to Yao and Zheng \cite{YZ}. Notice that in practice, there is a large gap between the positive definiteness of $Ha$ and the non-degeneracy of $Ha_m$. The above works in inhomogeneous polynomial case are not able to recover Theorem \ref{th0}, for the lack of subtlety in treating oscillatory integrals supported on a bounded domain, which the present paper also pays attention to even for non-polynomial phase functions.

Our approach for estimating (\ref{eq05}) can be viewed as a stationary phase method regarding two parameters $|t|$ and $|x|$, which is a refinement of the treatment for higher dimensional oscillatory integrals in \cite{KPG}. Following the philosophy of stationary phase, we will require $a(\xi)$ to have quantitative elliptic features and non-degeneracy in domains concerned. However, unlike previous studies mentioned above, we emphasize that $a(\xi)$ does not have to be polynomial even in higher dimensions, and $\psi$ does not have to be supported away from the origin. As it has been pointed out in \cite{KPG}, the ellipticity of $a$ can be removed under special choices of $\psi$. (See comments after Example \ref{ex1}.)

The rest of the paper is organized as follows. In Section 2, we study oscillatory integrals supported away from and near the origin in Lemma \ref{lm1} and Lemma \ref{lm2} respectively, where subtle decompositions quantified jointly in space-time are considered crucially. These two lemmas are established in settings more general than the ones we use later, and they both have partial sharpness, (see remarks after their proofs,) thus we expect their usage beyond this work. Section 3 is divided into two subsections for the applications of the previous technicality. In Subsection 3.1, we give our main result Theorem \ref{th1} for oscillatory integrals supported on $\mathbb{R}^n$, and two following examples are given to show its generality and sharpness. (As well as by an example given in the proof of Proposition \ref{cor3.2}.) In Subsection 3.2, $L^p-L^q$ type estimates of dispersive Cauchy problems are established for reasonable smoothing functions $\psi$ and $(p,q)$ pairs, and Strichartz type estimates are derived as a corollary. At the final, we also consider the $L^p$ type estimates for the fractional Schr\"{o}dinger equations concerning some complex-valued integrable potentials.

In the sequel, for $b\in\mathbb{R}$, $k\in\mathbb{N}_0:=\{0, 1,\cdots\}$, and an open set $\Omega\subset\mathbb{R}^n$, we define $S_k^b(\Omega)$ to be the class of all functions $f\in C^k(\Omega)$ such that
\begin{equation}\label{eq1}
|\partial^\alpha f(\xi)|\leq c_\alpha|\xi|^{b-|\alpha|}\quad\mathrm{if}~\xi\in\Omega~\mathrm{and}~\alpha\in\mathbb{N}_0^n~\mathrm{with}~|\alpha|\leq k,
\end{equation}
and put $S^b(\Omega)=\cap_{k=0}^{+\infty}S_k^b(\Omega)$. Moreover, we denote by $C\geq0$ a generic constant independent of $t,~x,~\xi$ whenever these letters occur.
\section{Oscillatory integrals}\label{sec2}
For any integer $K>\frac n2$ and real number $m>1$, we consider in this section oscillatory integrals of the following type:
\begin{equation}\label{eq2}
I(t,x)=\int_\Omega e^{i(ta(\xi)+x\cdot\xi)}\psi(\xi)d\xi,\quad t\in\mathbb{R}\setminus\{0\},~x\in\mathbb{R}^n,
\end{equation}
where $a\in S_{K+1}^m(\Omega)$ is real valued, and $\psi\in S_K^b(\mathbb{R}^n\setminus\{0\})$ for some $b\in\mathbb{R}$ with $\mathrm{supp}~\psi\subset\overline
\Omega$. Here (\ref{eq2}) is understood as the inverse Fourier transform of $(2\pi)^ne^{ita(\cdot)}\psi(\cdot)$ while it defines a temperate distribution in $\mathbb{R}^n$. With specific $\Omega$ and $b$, we shall prove that $I$ coincides with a locally integrable function in the $x$ variable, and satisfies point-wise estimates in different space-time domains.

We denote $\mu_b=\frac{n(m-2)-2b}{2(m-1)}$ in the rest of this paper, and always assume in this section that there exist constants $c_1,~c_2,~c_1',~c_2'$ such that
\begin{equation}\label{eq3}
c_1|\xi|^{m-1}\leq|\nabla a(\xi)|\leq c_2|\xi|^{m-1}\quad\mathrm{if}~\xi\in\Omega,
\end{equation}
and
\begin{equation}\label{eq4}
	c_1'|\xi|^{n(m-2)}\leq|\mathrm{det}Ha(\xi)|\leq c_2'|\xi|^{n(m-2)}\quad\mathrm{if}~\xi\in\Omega.
\end{equation}
\begin{lemma}\label{lm1}
	Let $\Omega=\{\xi\in\mathbb{R}^n;~|\xi|>r_0\}$ for some $r_0>0$. For any fixed $t_0>0$,\\
	i) if $b\in[\frac{n(m-2)}{2}-K(m-1),Km-n)$, then
	\begin{equation}\label{eq5}
		\begin{split}
			|I(t,x)|\leq\begin{cases}C|t|^{-\frac n2+\mu_b}|x|^{-\mu_b}&\text{if}~|t|\geq t_0,~|t|^{-1}|x|>\frac23c_1r_0^{m-1},\\
				C|t|^{-K}              &\text{if}~|t|\geq t_0,~|t|^{-1}|x|\leq\frac23c_1r_0^{m-1};
			\end{cases}
		\end{split}
	\end{equation}
	ii) if $b\in[-\frac n2,Km-n)$, then
	\begin{equation}\label{eq6}
		|I(t,x)|\leq C|t|^{-\frac{n+b}{m}}(1+|t|^{-\frac1m}|x|)^{-\mu_b}\quad\text{if}~0<|t|<t_0,~x\in\mathbb{R}^n.
	\end{equation}
\end{lemma}
\begin{proof}
For $\epsilon\in(0,1)$, define $\psi_\epsilon(\xi)=e^{-\epsilon|\xi|^2}\psi(\xi)$ and $I_\epsilon(t,x)$ with $\psi$ replaced by $\psi_\epsilon$ in (\ref{eq2}). We have for every $t\neq0$ that $I_\epsilon(t,\cdot)\in C(\mathbb{R}^n)$ and that $I_\epsilon(t,\cdot)\rightarrow I(t,\cdot)$ as $\epsilon\rightarrow0$ in the sense of distribution. It suffices to prove (\ref{eq5}) and (\ref{eq6}) for $I_\epsilon$ with constants $C$'s independent of $\epsilon$, since then $I(t,\cdot)$ is identified with an absolutely continuous Radon measure with the density $I$ satisfying each estimate. However, one checks that $\psi_\epsilon\in S_K^b(\mathbb{R}^n\setminus\{0\})$ has $K$ semi-norms defined in (\ref{eq1}) all bounded uniformly in $\epsilon$, and these bounds will be the only quantities we need for $\psi_\epsilon$. So it is reasonable to abuse our notations of $I_\epsilon$ and $\psi_\epsilon$ by using $I$ and $\psi$ instead to the end of this proof.
	
We first write $\Omega=\Omega_1\cup\Omega_2$ where
\begin{equation*}
\Omega_1=\{\xi\in\Omega;~\mbox{$|\nabla a(\xi)+\frac xt|<\frac12|\frac xt|$}\}
\end{equation*}
and
\begin{equation*}
\Omega_2=\{\xi\in\Omega;~\mbox{$|\nabla a(\xi)+\frac xt|>\frac14|\frac xt|$}\}.
\end{equation*}
Choose the following partition of unity subordinate to this covering:
\begin{equation*}
\varphi_1(\xi)=\varphi((\nabla a(\xi)+\mbox{$\frac xt$})/\mbox{$\frac12|\frac xt|$})~\mathrm{and}~\varphi_2(\xi)=1-\varphi_1(\xi),
\end{equation*}
where $\varphi\in C_c^\infty(\mathbb{R}^n)$ such that $\varphi(\xi)=1$ if $|\xi|\leq\frac12$ and $\varphi(\xi)=0$ if $|\xi|\geq1$. Observe that by (\ref{eq3}) we have
\begin{equation}\label{eq7}
\Omega_1\subset\{\xi\in\mathbb{R}^n;~(1/2c_2)^\frac{1}{m-1}r\leq|\xi|\leq(3/2c_1)^\frac{1}{m-1}r\},
\end{equation}
where $r:=|\frac xt|^{1/(m-1)}$, thus $|\partial^\alpha\varphi_j|\leq C|\xi|^{-|\alpha|}$ for $|\alpha|\leq K$, and $I(t,x)$ is split into
\begin{equation*}
I_j(t,x)=\int_{\Omega_j}e^{i(ta(\xi)+x\cdot\xi)}\varphi_j(\xi)\psi(\xi)d\xi,\quad j=1,2.
\end{equation*}
We shall prove (\ref{eq5}) and (\ref{eq6}) with $I$ replaced by $I_j$, $j=1,~2$.

{\bf Step 1: Estimates for $I_2$.} We write $\Omega_2=\bigcup_{j=1}^nU_j$ where
\begin{equation}\label{eq8}
U_j=\{\xi\in\Omega_2;~|\partial_ja(\xi)+\mbox{$\frac{x_j}{t}$}|>\mbox{$\frac{1}{\sqrt{2n}}$}|\nabla a(\xi)+\mbox{$\frac xt$}|\},
\end{equation}
and choose the following partition of unity of $\Omega_2$ subordinate to this covering: $\eta_j=\zeta_j/\Sigma_{l=1}^n\zeta_l$ ($j=1,~\cdots,~n$) where
\begin{equation*}
\zeta_j(\xi)=\zeta(\sqrt{2n}\left(\partial_j a(\xi)+\mbox{$\frac{x_j}{t}$}\right)\big/|\nabla a(\xi)+\mbox{$\frac xt$}|),\quad\xi\in\Omega_2,
\end{equation*}
while $\zeta\in C^\infty(\mathbb{R})$ such that $\zeta(s)=1$ if $|s|\geq1$ and $\zeta(s)=0$ if $|s|\leq\frac12$. Notice that in $\Omega_2$ we have
\begin{equation}\label{eq9}
|\nabla a(\xi)|<5|\nabla a(\xi)+\mbox{$\frac xt$}|,
\end{equation}
it follows that $|\partial^\alpha\eta_j|\leq C_\alpha|\xi|^{-|\alpha|}$ for $|\alpha|\leq K$ by induction, because the terms of $\partial^\alpha\zeta_l$ ($\alpha\neq0;~l=1,~\cdots,~n$) are of the form
\begin{equation}\label{eq10}
C\zeta^{(k)}\left(\frac{\sqrt{2n}(\partial_la+\frac{x_l}{t})}{|\nabla a+\frac xt|}\right)\frac{(\nabla a+\frac xt)^{\beta}}{|\nabla a+\frac xt|^{M+|\beta|}}\Pi_{j=1}^{M}(\partial^{\beta^{(j)}}a),
\end{equation}
where $k,~M\leq|\alpha|$; $\beta,~\beta^{(j)}\in\mathbb{N}^n$ ($j=1,~\cdots,~M$), $(\nabla a+\frac xt)^{\beta}:=\Pi_{j=1}^n(\partial_ja+\frac{x_j}{t})^{\beta_j}$ and $M+|\alpha|=\Sigma_{j=1}^{M}|\beta^{(j)}|$.
Now it suffices to estimate the integral
\begin{equation*}
I_{21}(t,x)=\int_{U_1}e^{i(ta(\xi)+x\cdot\xi)}\varphi_2(\xi)\eta_1(\xi)\psi(\xi)d\xi.
\end{equation*}
Set $D_*f:=\partial_1(gf)$ for $f\in C^K(\mathbb{R}^n)$, where $g=(it\partial_1a+ix_1)^{-1}$. Then by induction one obtains
\begin{equation*}
D_*^jf=\sum_{\alpha}a_\alpha(\partial_1^{\alpha_1}g)\cdots (\partial_1^{\alpha_j}g)(\partial_1^{\alpha_{j+1}}f),
\end{equation*}
for $j\leq K$, where the sum runs over all $\alpha=(\alpha_1,~\cdots,~\alpha_{j+1})\in\mathbb{N}_0^{j+1}$ such that $|\alpha|=j$ and $0\leq\alpha_1\leq\cdots\leq\alpha_j$. The terms of $\partial_1^jg$ are of the form
\begin{equation}\label{eq11}
Ct^{-1}\frac{(\partial_1^2a)^{m_2}\cdots(\partial_1^{j+1}a)^{m_{j+1}}}{(\partial_1a+\frac{x_1}{t})^{1+m_2+\cdots+m_{j+1}}},
\end{equation}
where $\Sigma_{l=2}^{j+1}(l-1)m_l=j$ and $\Sigma_{l=2}^{j+1}m_l\leq j$, thus by (\ref{eq9}) and (\ref{eq11}) one checks that for $j\leq K$ we have $|\partial_1^jg|\leq C|t|^{-1}|\xi|^{1-m-j}$, $|\partial_1^jg|\leq C|x|^{-1}|\xi|^{-j}$, and therefore
\begin{equation*}
|\partial_1^jg|\leq C|t|^{-(1-\theta)}|x|^{-\theta}|\xi|^{-(1-\theta)(m-1)-j},
\end{equation*}
for all $\theta\in[0,1]$. Notice that $|\partial_1^j(\varphi_2\eta_1\psi)|\leq C|\xi|^{b-j}$, we thus have
\begin{equation*}
|D_*^j(\varphi_2\eta_1\psi)|\leq C|t|^{-j(1-\theta)}|x|^{-j\theta}|\xi|^{b-j(1-\theta)(m-1)-j}.
\end{equation*}
Integration by parts for $K$ times leads to
\begin{equation}\label{eq12}
\begin{split}
|I_{21}|=&\left|\int_{U_1}e^{i(ta(\xi)+x\cdot\xi)}D_*^K(\varphi_2\eta_1\psi)d\xi\right|\\
\leq& C|t|^{-K(1-\theta)}|x|^{-K\theta}\int_{|\xi|>r_0}|\xi|^{b-K(1-\theta)(m-1)-K}d\xi.
\end{split}
\end{equation}
	
To prove (\ref{eq5}) for $I_2$, when $|t|\geq t_0$ and $|t|^{-1}|x|\leq2c_1r_0^{m-1}/3$, we choose $\theta=0$, so that $b-K(1-\theta)(m-1)-K=b-Km<-n$, and therefore
\begin{equation}\label{eq13}
|I_{21}|\leq C|t|^{-K}.
\end{equation}
When $|t|\geq t_0$ and $|t|^{-1}|x|>2c_1r_0^{m-1}/3$, notice that
\begin{equation*}
\begin{split}
&[\mbox{$\frac{n(m-2)}{2}$}-K(m-1),Km-n)\\
=&\bigcup_{\theta\in[0,1]}[-K(m-1)\theta+\mbox{$\frac{n(m-2)}{2}$},-K(m-1)\theta+Km-n),
\end{split}
\end{equation*}
there exists $\theta_0\in[0,1]$ such that
\begin{equation*}
\begin{cases}
b-K(1-\theta_0)(m-1)-K<-n,\\
\frac{b-K(1-\theta_0)(m-1)+\frac n2}{m-1}+K-\frac n2\geq0,
\end{cases}
\end{equation*}
thus we can choose $\theta=\theta_0$, and have by (\ref{eq12}), and by the assumption $K>\frac n2$, that
\begin{equation}\label{eq14}
\begin{split}
|I_{21}|\leq& C|t|^{-\frac n2+\mu_b}|x|^{-\mu_b}|t|^{-(K-\frac n2)}|\mbox{$\frac xt$}|^{-\left(\frac{b-K(1-\theta_0)(m-1)+\frac n2}{m-1}+K-\frac n2\right)}\\
\leq& C|t|^{-\frac n2+\mu_b}|x|^{-\mu_b}.
\end{split}
\end{equation}
	
To prove (\ref{eq6}) for $I_2$, we split $I_{21}$ into
\begin{equation*}
\begin{split}
I_{21}(t,x)=& I_{211}(t,x)+I_{212}(t,x)\\
=&\int_{U_1}e^{i(ta(\xi)+x\cdot\xi)}\widetilde{\varphi}(\xi)\varphi_2(\xi)\eta_1(\xi)\psi(\xi)d\xi\\
&\quad+\int_{U_1}e^{i(ta(\xi)+x\cdot\xi)}(1-\widetilde{\varphi}(\xi))\varphi_2(\xi)\eta_1(\xi)\psi(\xi)d\xi,
\end{split}
\end{equation*}
where $\widetilde{\varphi}(\xi)=\varphi(\frac{\xi}{r_1})$, and $r_1$ is to be chosen in different domains of $(t,x)$. We have $|\partial_1^j\widetilde{\varphi}(\xi)|\leq C|\xi|^{-j}$ and thus $|\partial_1^j((1-\widetilde{\varphi}(\xi))\varphi_2(\xi)\eta_1(\xi)\psi(\xi))|\leq C|\xi|^{b-j}$ for $j\leq K$. Obviously
\begin{equation}\label{eq15}
|I_{211}|\leq C\int_{|\xi|\leq r_1}|\xi|^bd\xi\leq Cr_1^{n+b}.
\end{equation}
Similar to (\ref{eq12}), we also have
\begin{equation}\label{eq16}
\begin{split}
|I_{212}|=&\left|\int_{U_1}e^{i(ta(\xi)+x\cdot\xi)}D_*^K((1-\widetilde{\varphi})\varphi_2\eta_1\psi)d\xi\right|\\
\leq& C|t|^{-K(1-\theta)}|x|^{-K\theta}r_1^{b-K(1-\theta)(m-1)-K+n}\int_{|\xi|>\frac12}|\xi|^{b-K(1-\theta)(m-1)-K}d\xi.
\end{split}
\end{equation}
When $0<|t|<t_0$ and $|t|^{-\frac1m}|x|>1$, we choose $r_1=|t|^{-\frac{n+2b}{2(n+b)(m-1)}}|x|^{-\frac{\mu_b}{n+b}}$ in (\ref{eq15}) and (\ref{eq16}). Observe that
\begin{equation*}
[\mbox{$-\frac n2$},Km-n)=\bigcup_{\theta\in[0,1]}[\mbox{$-\frac n2+\frac{n(m-1)(1-\theta)}{2}$},K(1-\theta)(m-1)+K-n),
\end{equation*}
there exists $\theta_1\in[0,1]$ such that
\begin{equation*}
\begin{cases}
b-K(1-\theta_1)(m-1)-K<-n,\\
2b+n-n(m-1)(1-\theta_1)\geq0,
\end{cases}
\end{equation*}
then we choose $\theta=\theta_1$ in (\ref{eq16}). Now by (\ref{eq15}) and (\ref{eq16}),
\begin{equation}\label{eq17}
\begin{split}
|I_{21}|\leq&|I_{211}|+|I_{212}|\\
\leq& C|t|^{-\frac{n+2b}{2(m-1)}}|x|^{-\mu_b}+C|t|^{-\frac{n+2b}{2(m-1)}}|x|^{-\mu_b}(|t|^{-\frac1m}|x|)^{-\frac{Km(2b+n-n(m-1)(1-\theta_1))}{2(m-1)(n+b)}}\\
\leq& C|t|^{-\frac{n+2b}{2(m-1)}}|x|^{-\mu_b}\\
\leq& C|t|^{-\frac{n+b}{m}}(1+|t|^{-\frac1m}|x|)^{-\mu_b}.
\end{split}
\end{equation}
When $0<|t|<t_0$ and $|t|^{-\frac1m}|x|\leq1$, we choose $r_1=|t|^{-\frac1m}$ and $\theta=0$ in (\ref{eq15}) and (\ref{eq16}). Since $b-K(1-\theta)(m-1)-K=b-Km<-n$, we have
\begin{equation}\label{eq18}
\begin{split}
|I_{21}|\leq&|I_{211}|+|I_{212}|\\
\leq& C|t|^{-\frac{n+b}{m}}\\
\leq& C|t|^{-\frac{n+b}{m}}(1+|t|^{-\frac1m}|x|)^{-\mu_b}.
\end{split}
\end{equation}
Now (\ref{eq13}) and (\ref{eq14}) prove (\ref{eq5}) for $I_2$, while (\ref{eq17}) and (\ref{eq18}) prove (\ref{eq6}) for $I_2$.

{\bf Step 2: Estimates for $I_1$.} We first need some preparation. Denote by $c,~c_j$ $(j=3,~4,~\cdots)$ the absolute constants depending only on $n,~m,~c_1,~c_2,~c_1',~c_2'$ and $c_\alpha$ (see (\ref{eq1}), (\ref{eq3}) and (\ref{eq4})). Recall (\ref{eq7}) and denote $c_3=(1/2c_2)^{1/(m-1)}$, $c_4=(3/2c_1)^{1/(m-1)}$, we may assume that $r>r_0/c_4$, which means $|t|^{-1}|x|>2c_1r_0^{m-1}/3$, since otherwise $I_1=0$.
	
Set $\Omega'=\{\xi\in\mathbb{R}^n;~c_3r<|\xi|<2c_4r\}$. While $n>1$ ($\Rightarrow K\geq2$), we observe that if a line segment $[\xi',\xi'']$ lies in $\Omega\cap\Omega'$, then by Taylor's formula
\begin{equation*}
\nabla a(\xi')=\nabla a(\xi'')+Ha(\xi'')(\xi'-\xi'')+R(\xi',\xi''),
\end{equation*}
where
\begin{equation*}
\begin{split}
|R(\xi',\xi'')|=&\left|\int_0^1a^{(3)}(\xi''+s(\xi'-\xi'');\xi'-\xi'',\xi'-\xi'')(1-s)ds\right|\\
\leq&\mbox{$\frac12$}c_5r^{m-3}|\xi'-\xi''|^2.
\end{split}
\end{equation*}
Here $a^{(3)}$ refers to the third order differential of $a$ which is a tri-linear operator (see \cite[p. 11]{H}). If we also have $[\xi,\xi'']\subset\Omega\cap\Omega'$ with $|\xi-\xi''|=|\xi'-\xi''|\leq|\xi-\xi'|$, it follows from (\ref{eq4}) that
\begin{equation*}
\begin{split}
|\nabla a(\xi)-\nabla a(\xi')|\geq&|Ha(\xi'')(\xi-\xi')|-|R(\xi,\xi'')|-|R(\xi',\xi'')|\\
\geq&\min_{1\leq j\leq n}|\lambda_j(\xi'')||\xi-\xi'|-c_5r^{m-3}|\xi-\xi'|^2\\
\geq&2cr^{m-2}|\xi-\xi'|-c_5r^{m-3}|\xi-\xi'|^2,
\end{split}
\end{equation*}
where $\{\lambda_j\}_{j=1}^n$ are the eigenvalues of $Ha$, and the last line comes from the fact that $|\lambda_j(\xi'')|$ is controlled from above by the spectral radius of $Ha(\xi'')$, thus by any consistent matrix norm of $Ha(\xi'')$ which has the bound $C|\xi''|^{n(m-2)}$ since $a\in S_K^m(\Omega)$, and that $\lambda_j=\mathrm{det}Ha/\Pi_{i\neq j}\lambda_i$ for each $j$.
	
Consider the conic and annular decomposition of $\Omega_1$ to apply the above inequality. Choose a finite set $\{\xi_v\}\subset S^{n-1}$ (the unit sphere in $\mathbb{R}^n$) such that $|\xi_v-\xi_{v'}|\geq c_6$ for $v\neq v'$ and $\min_{v}|\xi-\xi_v|<c_6$ for all $\xi\in S^{n-1}$, where $c_6$ will be chosen below. Corresponding to $\{\xi_v\}$, write $\Omega_1=\cup_v\cup_{k=2}^{k_0}(\Omega^{vk}\cap\Omega_1)$ where $\Omega^{vk}=r\widetilde{\Omega}^{vk}:=\{r\xi\in\mathbb{R}^n;~\xi\in\widetilde{\Omega}^{vk}\}$,
\begin{equation*}
\widetilde{\Omega}^{vk}=\{\xi\in \mathbb{R}^n;~|\mbox{$\xi\over|\xi|$}-\xi_v|<2c_6~\mathrm{and}~(c_3+kc_7-2c_7)<|\xi|<(c_3+kc_7)\},
\end{equation*}
and $k_0\in\mathbb{N}$ is so large that $c_7:=(c_4-c_3)/k_0\leq c/4c_5$. For the geometric solid $\widetilde{\Omega}^{vk}$, $c_6$ can be chosen small enough such that for all $v$ and $k$ we have $\mathrm{dist}(\widetilde{\Omega}^{vk})\leq c/c_5$, and that for any $\xi,~\xi'\in\Omega^{vk}$ there exists $\xi''\in\Omega\cap\Omega'$ such that $[\xi,\xi'']\cup[\xi',\xi'']\subset\Omega\cap\Omega'$ with $|\xi-\xi''|=|\xi'-\xi''|\leq|\xi-\xi'|$. As a reminder, the constant $2$ in the definition of $\Omega'$ is to insure the existence of such $\xi''$ in the cases where $r_0\in r(c_3+(k_0-2)c_7,c_3+k_0c_7)$. Therefore
\begin{equation}\label{eq19}
|\nabla a(\xi)-\nabla a(\xi')|\geq cr^{m-2}|\xi-\xi'|\quad\mathrm{if}~\xi,\xi'\in\Omega^{vk}.
\end{equation}
While $n=1$, (\ref{eq19}) still holds from the mean value theorem. Meanwhile we choose a partition of unity $\{\widetilde{\chi}_{vk}\}$ subordinate to the covering $\{\widetilde{\Omega}^{vk}\}$ of $r^{-1}\Omega_1$, so that $\{\chi_{vk}(\xi):=\widetilde{\chi}_{vk}(r^{-1}\xi)\}$ is a partition of unity of $\Omega_1$ subordinate to $\{\Omega^{vk}\}$. Then $|\partial^\alpha\chi_{vk}|\leq C|\xi|^{-|\alpha|}$ for $|\alpha|\leq K$ by (\ref{eq7}). We write $I=\Sigma_{v,k}I_1^{vk}$ where
\begin{equation*}
I_1^{vk}(t,x)=\int_{\Omega^{vk}}e^{i(ta(\xi)+x\cdot\xi)}\varphi_1(\xi)\chi_{vk}(\xi)\psi(\xi)d\xi.
\end{equation*}
Notice that the number of the indices $v$ and $k$ is finite and universal in $(t,x)$.\\
	
Now it suffices to estimate $I_1^{vk}$. We pick up $\xi_0\in\Omega^{vk}\cap\Omega_1$ such that $|\nabla a(\xi_0)+\frac xt|\leq\frac c4r^{m-2}r_2$ where $r_2=|t|^{-1/2(m-1)}|x|^{(2-m)/2(m-1)}$, while the case when such $\xi_0$ does not exist can be easily treated (see \cite[p. 53]{KPG}). Subordinate to the supports $V_1=\{\xi\in\Omega^{vk};~|\xi-\xi_0|<r_2\}$ and $V_2=\{\xi\in\Omega^{vk};~|\xi-\xi_0|>\frac12r_2\}$ we split $I_1^{vk}$ into $I_{11}^{vk}$ and $I_{12}^{vk}$ smoothly (see $\eta$ below). It's obvious that
\begin{equation}\label{eq20}
|I_{11}^{vk}|\leq Cr^br_2^n=C|t|^{-\frac n2+\mu_b}|x|^{-\mu_b}.
\end{equation}
	
Replace $\Omega_2$ by $V_2$ in (\ref{eq8}) and denote these sets by $W_j$, we split $I_{12}^{vk}$ into $n$ new integrals $I_{12j}^{vk}$, still naming the corresponding partition functions $\eta_j$. To estimate $I_{12}^{vk}$, similarly, it suffices to estimate
\begin{equation*}
I_{121}^{vk}(t,x)=\int_{W_1}e^{i(ta(\xi)+x\cdot\xi)}D_*^K\eta(\xi)d\xi,
\end{equation*}
where
\begin{equation*}
\eta(\xi)=\varphi_1(\xi)\chi_{vk}(\xi)(1-\varphi((\xi-\xi_0)/r_2))\eta_1(\xi)\psi(\xi).
\end{equation*}
By (\ref{eq7}) and (\ref{eq11}) one first checks
\begin{equation*}
|\partial_1^jg|\leq C|t|^{-1}|\xi|^{j(m-2)}|\partial_1a(\xi)+\mbox{$\frac{x_1}{t}$}|^{-j-1}.
\end{equation*}
Further, (\ref{eq7}) and (\ref{eq19}) imply,
\begin{equation*}
\begin{split}
|\nabla a(\xi)+\mbox{$\frac xt$}|\geq&|\nabla a(\xi)-\nabla a(\xi_0)|-|\nabla a(\xi_0)+\mbox{$\frac xt$}|\\
\geq& cr^{m-2}|\xi-\xi_0|-\mbox{$\frac c4$}r^{m-2}r_2\\
\geq&\mbox{$\frac c2(\frac{|\xi|}{c_4})^{m-2}$}|\xi-\xi_0|,
\end{split}
\end{equation*}
we thus have
\begin{equation*}
|\partial_1^jg|\leq C|t|^{-1}|\xi|^{2-m}|\xi-\xi_0|^{-j-1}.
\end{equation*}
Moreover, by (\ref{eq7}) and (\ref{eq10}), we have
\begin{equation*}
|\partial_1^j\eta_1|\leq C|\xi|^{j(m-2)}|\nabla a(\xi)+\mbox{$\frac xt$}|^{-j}\leq C|\xi-\xi_0|^{-j}.
\end{equation*}
These estimates, together with $\mathrm{dist}(\Omega^{vk})\leq cr/c_5$ and (\ref{eq7}), imply $|\partial_1^j\eta|\leq C|\xi|^b|\xi-\xi_0|^{-j}$, and consequently,
\begin{equation*}
|D_*^K\eta|\leq C|t|^{-K}r^{b+K(2-m)}|\xi-\xi_0|^{-2K}.
\end{equation*}
Therefore by $K>\frac n2$ we have
\begin{equation}\label{eq21}
\begin{split}
|I_{121}^{vk}|\leq&\int_{|\xi-\xi_0|\geq\frac12r_2}|D_*^K\eta(\xi)|d\xi\\
\leq& C|t|^{-K}r^{b+K(2-m)}r_2^{n-2K}\int_{|\xi|\geq\frac12}|\xi|^{-2K}d\xi\\
\leq& C|t|^{-\frac n2+\mu_b}|x|^{-\mu_b}.
\end{split}
\end{equation}
	
Now (\ref{eq20}) and (\ref{eq21}), together with the fact that $I_1=0$ when $r\leq r_0/c_4$, prove (\ref{eq5}) for $I_1$. To prove (\ref{eq6}) for $I_1$, first notice that when $0<|t|<t_0$ and $|t|^{-\frac1m}|x|>1$, (\ref{eq20}) and (\ref{eq21}) also prove
\begin{equation*}
|I_1|\leq C|t|^{-\frac n2+\mu_b}|x|^{-\mu_b}\leq C|t|^{-\frac{n+b}{m}}(1+|t|^{-\frac1m}|x|)^{-\mu_b}.
\end{equation*}
When $0<|t|<t_0$ and $|t|^{-\frac1m}|x|\leq1$, by (\ref{eq7}) and that $b\geq-\frac n2$, we have
\begin{equation*}
\begin{split}
|I_1|\leq&\int_{c_3r\leq|\xi|\leq c_4r}|\xi|^bd\xi\\
\leq& C|t|^{-\frac{n+b}{m-1}}|x|^{\frac{n+b}{m-1}}\\
\leq& C|t|^{-\frac{n+b}{m}}(1+|t|^{-\frac1m}|x|)^{-\mu_b}.
\end{split}
\end{equation*}
This completes the proof.
\end{proof}
\begin{remark}\label{rk0}
Lemma \ref{lm1} is sharp in the domain
\begin{equation*}
\{(t,x)\in\mathbb{R}^{n+1};~|t|\geq t_0,~|t|^{-1}|x|\geq N\}
\end{equation*}
for some sufficiently large $N>0$. In fact, \cite[Lemma 1]{Mi} considers the case that $a(\xi)=|\xi|^m$, $\psi(\xi)=\phi(\xi)(1+|\xi|^2)^{b/2}$ where $\phi$ is some cutoff function supported away from the origin, and gains in such domain the lower bound for $|I|$, which is the same as the upper bound in (\ref{eq5}). See another proof after Example \ref{ex2}.
\end{remark}
\begin{remark}\label{rk00}
Lemma \ref{lm1} is also sharp in the domain
\begin{equation*}
\{(t,x)\in\mathbb{R}^{n+1};~0<|t|\leq t_0,~|t|^{-\frac1m}|x|\geq \tau_0\}
\end{equation*}
for some sufficiently large $\tau_0>0$. See comments after Example \ref{ex2}.
\end{remark}
\begin{remark}\label{rk1}
For any $\delta\in(0,\frac12)$, if we write $\Omega=\Omega_1\cup\Omega_2$ where
\begin{equation*}
\Omega_1=\{\xi\in\Omega;~\mbox{$|\nabla a(\xi)+\frac xt|<\delta|\frac xt|$}\}
\end{equation*}
and
\begin{equation*}
\Omega_2=\{\xi\in\Omega;~\mbox{$|\nabla a(\xi)+\frac xt|>\frac\delta 2|\frac xt|$}\},
\end{equation*}
with minor modification in the proof of Lemma \ref{lm1}, (\ref{eq5}) can be refined as
\begin{equation*}
\begin{split}
|I(t,x)|\leq\begin{cases}C|t|^{-\frac n2+\mu_b}|x|^{-\mu_b}&\text{if}~|t|\geq t_0,~|t|^{-1}|x|>\frac{c_1r_0^{m-1}}{1+\delta};\\
C|t|^{-K}              &\text{if}~|t|\geq t_0,~|t|^{-1}|x|\leq\frac{c_1r_0^{m-1}}{1+\delta}.
\end{cases}
\end{split}
\end{equation*}
\end{remark}
\begin{remark}\label{rk2}
The conditions of $b$ in Lemma \ref{lm1} mainly arise from estimating $I_2$. In fact, one realizes that to prove (\ref{eq5}) and (\ref{eq6}) for $I_1$, we only need $b\geq-n$. We also point out that relaxing the range of $b$ is possible for smaller space-time domains by same kinds of considerations, (as well as for Lemma \ref{lm2} below), but this will only increase the length of argument which we don't aim at.
\end{remark}
\begin{remark}
When $m=1$, our proof completely fails, and we refer to \cite{Mi2} for relevant topics. When $0<m<1$, it seems that most details become sensitive, (also for Lemma \ref{lm2} below,) and we refer to \cite{Mi3} for a comparison.
\end{remark}

The following is the other setting we take interest of.

\begin{lemma}\label{lm2}
Let $\Omega=\{\xi\in\mathbb{R}^n\setminus\{0\};~|\xi|<r_0\}$ for some $r_0>0$. For any fixed $\tau_0>0$,\\
i) if $b\in(-n,K-n)$, then
\begin{equation}\label{eq22}
|I(t,x)|\leq C(1+|x|)^{-(n+b)}\quad\text{if}~|t|^{-\frac1m}|x|>\tau_0,~|t|^{-1}|x|>2c_2r_0^{m-1};
\end{equation}
ii) if $b\in[-\frac n2,Km-n)$, then
\begin{equation}\label{eq23}
|I(t,x)|\leq C|t|^{-\frac n2+\mu_b}|x|^{-\mu_b}\quad\text{if}~|t|^{-\frac1m}|x|>\tau_0,~|t|^{-1}|x|\leq2c_2r_0^{m-1};
\end{equation}
iii) if $b\in(-n,Km-n)$, then
\begin{equation}\label{eq24}
|I(t,x)|\leq C(1+|t|^\frac1m)^{-(n+b)}            \quad\text{if}~|t|^{-\frac1m}|x|\leq\tau_0.
\end{equation}
\end{lemma}
\begin{proof}
The proof is somehow parallel to that of Lemma \ref{lm1}, and we start by splitting $I=I_1+I_2$ in exactly the same way. Notice that $I(t,\cdot)\in C^\infty(\mathbb{R}^n)$ in all cases.

{\bf Step 1: Estimates for $I_1$.} As in the proof of Lemma \ref{lm1}, we recall that (\ref{eq7}) holds and use the same notations of $c_3$ and $c_4$. We may assume that $r<r_0/c_3$, which means $|t|^{-1}|x|<2c_2r_0^{m-1}$, since otherwise $I_1=0$.
	
If $\{\widetilde{\Omega}^v\}$ is a finite open covering of $\{\xi\in\mathbb{R}^n;~c_3\leq|\xi|\leq c_4\}$, such that for every $v$ and any $\xi,~\xi'\in\widetilde{\Omega}^v$ we have $0\notin[\xi,\xi']$ and $\mathrm{dist}(\widetilde{\Omega}^v)\leq\epsilon$ where $\epsilon>0$ will be chosen below, we can write $\Omega_1=\cup_v(\Omega^v\cap\Omega_1)$ where $\Omega^v=r\widetilde{\Omega}^v$. Now for every $v$ and any $\xi,~\xi'\in\Omega^v\cap\Omega_1$, we have $0\notin[\xi,\xi']\subset\Omega$, $|\xi-\xi'|\leq\epsilon r$ and
\begin{equation*}
\begin{split}
|\nabla a(\xi)-\nabla a(\xi')|\geq&|Ha(\xi')(\xi-\xi')|-|R(\xi,\xi')|\\
\geq&2cr^{m-2}|\xi-\xi'|-c_5r^{m-3}|\xi-\xi'|^2\\
\geq& cr^{m-2}|\xi-\xi'|,
\end{split}
\end{equation*}
if $\epsilon$ is chosen small. By this fact and the same argument in the proof of Lemma \ref{lm1}, (\ref{eq21}) holds when $|t|^{-1}|x|<2c_2r_0^{m-1}$, which proves (\ref{eq23}) for $I_1$. When $|t|^{-1}|x|<2c_2r_0^{m-1}$ and $|t|^{-\frac1m}|x|\leq\tau_0$, we have
\begin{equation*}
\begin{split}
|I_1|\leq&\int_{c_3r\leq|\xi|\leq c_4r}|\xi|^bd\xi\\
\leq& C|t|^{-\frac{n+b}{m-1}}|x|^{\frac{n+b}{m-1}}\\
\leq& C(1+|t|^\frac1m)^{-(n+b)},
\end{split}
\end{equation*}
because $b>-n$. This, and the fact that $I_1=0$ when $|t|^{-1}|x|\geq2c_2r_0^{m-1}$, prove (\ref{eq24}) for $I_1$, and then (\ref{eq22}) holds for $I_1$ trivially. We again point out that $b>-n$ is the only condition on $b$ we need in the arguments for $I_1$.

{\bf Step 2: Estimates for $I_2$.} Similar to those in the proof of Lemma \ref{lm1}, we choose $0<r_1=r_1(|t|,|x|)$ and accordingly split $I_{21}=I_{211}+I_{212}$ when $r_1<2r_0$. Then $I_{211}$ satisfies (\ref{eq15}) since $b>-n$, while $I_{212}$ satisfies (\ref{eq16}) for some $\theta\in[0,1]$ such that $b<K(1-\theta)(m-1)+K-n$. By always choosing $r_1=(|t|^{\theta-1}|x|^{-\theta})^{1/((1-\theta)(m-1)+1)}$, we have
\begin{equation*}
|I_{21}|\leq|I_{211}|+|I_{212}|\leq Cr_1^{n+b}\quad\mathrm{if}~0<r_1<2r_0.
	\end{equation*}
	However
	\begin{equation*}
		|I_{21}|\leq\int_{|\xi|<r_0}|\xi|^bd\xi\leq C,
	\end{equation*}
	therefore
	\begin{equation}\label{eq25}
		|I_{21}|\leq C(1+r_1^{-1})^{-(n+b)}.
	\end{equation}
	
	Choosing $\theta=0$, and consequently having $b<Km-n$, prove (\ref{eq24}) for $I_2$. Choosing $\theta=1$, and consequently having $b<K-n$, prove (\ref{eq22}) for $I_2$; meanwhile (\ref{eq23}) for $I_2$ when $b\in(-\frac n2,K-n)$ is also proved, since $|t|^{-\frac1m}|x|>\tau_0$ and $|t|^{-1}|x|\leq2c_2r_0^{m-1}$ imply $|x|>2c_2(\frac{\tau_0}{2c_2})^\frac{m}{m-1}r_0^{-1}$, and we have
	\begin{equation*}
		\begin{split}
			(1+|x|)^{-(n+b)}\leq&C|x|^{-(n+b)}\\
			=&C|t|^{-\frac n2+\mu_b}|x|^{-\mu_b}(|t|^{-\frac 1m}|x|)^{-\frac{m(n+2b)}{2(m-1)}}\\
			\leq&C|t|^{-\frac n2+\mu_b}|x|^{-\mu_b}.
		\end{split}
	\end{equation*}
	
	Finally, to prove (\ref{eq23}) for $I_2$ when $b\in[K-n,Km-n)$, notice that $\frac{Km-n-b}{K(m-1)}\in(0,1]$, we choose $\theta=\frac{Km-n-b-\epsilon}{K(m-1)}>0$ for sufficiently small $\epsilon>0$, thus $b<K(1-\theta)(m-1)+K-n$ holds. When $|t|^{-\frac1m}|x|>\tau_0$ and $|t|^{-1}|x|\leq2c_2r_0^{m-1}$, there exists $C_\theta>0$ such that $r_1<C_\theta$, then (\ref{eq25}) reads
\begin{equation*}
\begin{split}
|I_{21}|\leq&Cr_1^{n+b}\\
=&C(|t|^{\theta-1}|x|^{-\theta})^\frac{n+b}{(1-\theta)(m-1)+1}\\
=&C|t|^{-\frac n2+\mu_b}|x|^{-\mu_b}(|t|^{-\frac1m}|x|)^{-\frac{(2K-n)m}{2(m-1)}+\frac{Km\epsilon}{(m-1)(n+b+\epsilon)}}|t|^{-\frac{(n+b-K+\epsilon)\epsilon}{n+b-K+\epsilon}}.
		\end{split}
\end{equation*}
	Also notice that $|t|^{-\frac1m}|x|>\tau_0$ and $|t|^{-1}|x|\leq2c_2r_0^{m-1}$ imply $|t|>(\frac{\tau_0}{2c_2})^\frac{m}{m-1}r_0^{-m}$, so by choosing $\epsilon$ small and recall $K>\frac n2$, we have
	\begin{equation*}
		|I_{21}|\leq C|t|^{-\frac n2+\mu_b}|x|^{-\mu_b},
	\end{equation*}
	which completes the proof.
\end{proof}
\begin{remark}\label{rk3}
Lemma \ref{lm2} is sharp in the domain
\begin{equation*}
\{(t,x)\in\mathbb{R}^{n+1};|t|^{-1}|x|< N,~|t|^{-\frac1m}|x|\geq\tau_0\}
\end{equation*}
for sufficiently small $N$ and sufficiently large $\tau_0$. See comments after Example \ref{ex2}.
\end{remark}
\begin{remark}\label{rk4}
Similar to Remark \ref{rk1}, the space-time domains in (\ref{eq22}) and (\ref{eq23}) can be refined as $|t|^{-\frac 1m}|x|>\tau_0$, $|t|^{-1}|x|>\frac{c_2r_0^{m-1}}{1-\delta}$ and $|t|^{-\frac 1m}|x|>\tau_0$, $|t|^{-1}|x|\leq\frac{c_2r_0^{m-1}}{1-\delta}$ respectively, for any $\delta\in(0,\frac12)$.
\end{remark}

\section{Main results and applications}\label{sec3}

\subsection{Main results}\label{sec3.1}
Combining Lemma \ref{lm1} and Lemma \ref{lm2} leads to the following theorem concerning oscillatory integrals supported on $\mathbb{R}^n$.

\begin{theorem}\label{th1}
Let $1<m_1\leq m_2$, $-\frac n2\leq b_1\leq b_2$, and $\nu_j=\frac{n(m_j-2)-2b_j}{2(m_j-1)}$ for $j=1,~2$. Suppose $a,~\psi\in C^\infty(\mathbb{R}^n\setminus\{0\})$, $a$ is real valued, and there exist $0<r_0<R_0$ such that $a\in S^{m_j}(B_j)$ and $\psi\in S^{b_j}(B_j)$ for $j=1,~2$, where $B_1=\{\xi\in\mathbb{R}^n\setminus\{0\};~|\xi|<R_0\}$ and $B_2=\{\xi\in\mathbb{R}^n;~|\xi|>r_0\}$. Further, assume that (\ref{eq4}) holds in $B_j$ for $m=m_j$ ($j=1,~2$), and there exist $d_1,~d_1',~d_2,~d_2'>0$ satisfying $d_1r_0^{m_2-1}\leq d_2R_0^{m_1-1}$, such that
\begin{align}\label{equ3.1}
d_2'|\xi|^{m_1-1}\leq|\nabla a(\xi)|\leq d_2|\xi|^{m_1-1},\quad\xi\in B_1,
\end{align}
and
\begin{align}\label{equ3.2}
d_1|\xi|^{m_2-1}\leq|\nabla a(\xi)|\leq d_1'|\xi|^{m_2-1},\quad\xi\in B_2.
\end{align}
Consider (\ref{eq2}) where $\Omega=\mathbb{R}^n\setminus\{0\}$, then for any fixed $t_0,~N>0$, we have
\begin{equation}\label{eq26}
\begin{split}
|I(t,x)|\leq\begin{cases}C|t|^{-\frac{n+b_2}{m_2}}(1+|t|^{-\frac{1}{m_2}}|x|)^{-\nu_2}&\text{if}~0<|t|< t_0,~x\in\mathbb{R}^n~\text{or}~|t|\geq t_0,~|t|^{-1}|x|>N;\\
C|t|^{-\frac{n+b_1}{m_1}}(1+|t|^{-\frac{1}{m_1}}|x|)^{-\nu_1}&\text{if}~|t|\geq t_0,~|t|^{-1}|x|\leq N.
\end{cases}
\end{split}
\end{equation}
\end{theorem}
\begin{proof}
Choose $\chi\in C_c^\infty(\mathbb{R}^n)$ such that $\mathrm{supp}~\chi\subset\overline{B_1}$ and $\mathrm{supp}~(1-\chi)\subset\overline{B_2}$, then $\chi\psi\in S^{b_1}(B_1)$ and $(1-\chi)\psi\in S^{b_2}(B_2)$. Splitting $I$ in such way, and using $d_1r_0^{m_2-1}\leq d_2R_0^{m_1-1}$, it's easy to apply Lemma \ref{lm1} and Lemma \ref{lm2} to get
\begin{equation*}
\begin{split}
|I(t,x)|\leq\begin{cases}C|t|^{-\frac{n+b_2}{m_2}}(1+|t|^{-\frac{1}{m_2}}|x|)^{-\nu_2}&\mathrm{if}~0<|t|<t_0,~x\in\mathbb{R}^n;\\
C|t|^{-\frac n2+\nu_2}|x|^{-\nu_2}&\mathrm{if}~|t|\geq t_0,~|t|^{-1}|x|>2d_2R_0^{m_1-1};\\
C|t|^{-\frac{n+b_1}{m_1}}(1+|t|^{-\frac{1}{m_1}}|x|)^{-\nu_1}&\mathrm{if}~|t|\geq t_0,~|t|^{-1}|x|\leq 2d_2R_0^{m_1-1}.
\end{cases}
\end{split}
\end{equation*}
But observe that
\begin{equation*}
|t|^{-\frac n2+\nu_2}|x|^{-\nu_2}\approx|t|^{-\frac n2+\nu_1}|x|^{-\nu_1}\approx|t|^{-\frac{n+b_1}{m_1}}(1+|t|^{-\frac{1}{m_1}}|x|)^{-\nu_1}
\end{equation*}
when $|t|\geq t_0$ and $|t|^{-1}|x|$ is bounded from below and above; and also observe that
\begin{equation*}
|t|^{-\frac n2+\nu_2}|x|^{-\nu_2}\approx|t|^{-\frac{n+b_2}{m_2}}(1+|t|^{-\frac{1}{m_2}}|x|)^{-\nu_2}
\end{equation*}
when $|t|\geq t_0$ and $|t|^{-1}|x|$ is bounded from below. Thus the conclusion holds.
\end{proof}

The above theorem gives the point-wise estimates of a large class of dispersive kernels, as well as of their derivatives. We are now going to show its generality and sharpness through some examples.

\begin{example}\label{ex1}
If $a(\xi)=\Sigma_{j=1}^JA_j|\xi|^{m_j}$ where $A_1,~A_J>0$, $A_j\geq0$, and $1<m_1\leq\cdots\leq m_J$, then $a\in S^{m_1}(B_1)$, $a\in S^{m_J}(B_2)$ with (\ref{eq4}) holds respectively. Calculus shows that
\begin{equation*}
\begin{split}
\begin{cases}
A_1m_1|\xi|^{m_1-1}\leq|\nabla a(\xi)|\leq(A_1m_1+\Sigma_{j>1}A_jm_jR_0^{m_j-m_1})|\xi|^{m_1-1}\quad\text{if}~\xi\in B_1,\\
A_Jm_J|\xi|^{m_J-1}\leq|\nabla a(\xi)|\leq(\Sigma_{j<J}A_jm_jr_0^{m_j-m_J}+A_Jm_J)|\xi|^{m_J-1}\quad\text{if}~\xi\in B_2.
\end{cases}
\end{split}
\end{equation*}
Thus by Theorem \ref{th1} when choosing $\psi(\xi)=\xi^\alpha$ we have
\begin{equation*}
\begin{split}
&\left|\partial_x^\alpha\int_{\mathbb{R}^n}e^{i(ta(\xi)+x\cdot\xi)}d\xi\right|\\
\leq&\begin{cases}C|t|^{-\frac{n+|\alpha|}{m_J}}(1+|t|^{-\frac{1}{m_J}}|x|)^{-\frac{n(m_J-2)-2|\alpha|}{2(m_J-1)}}&\text{if}~0<|t|<t_0,~x\in\mathbb{R}^n~\text{or}~|t|\geq t_0,~|t|^{-1}|x|>N;\\
C|t|^{-\frac{n+|\alpha|}{m_1}}(1+|t|^{-\frac{1}{m_1}}|x|)^{-\frac{n(m_1-2)-2|\alpha|}{2(m_1-1)}}&\text{if}~|t|\geq t_0,~|t|^{-1}|x|\leq N,
\end{cases}
\end{split}
\end{equation*}
for $\alpha\in\mathbb{N}_0^n$ and any fixed $t_0,~N>0$.
\end{example}

As mentioned in the Introduction, Example \ref{ex1} recovers Theorem \ref{th0} by choosing $a(\xi)=|\xi|^4+|\xi|^2$, $t_0=1$ and $N=1$. In \cite{KPG}, the case $a(\xi)=\Sigma_{j=1}^n\pm|\xi_j|^{m_j}$ was also considered, and it's clear we can also generalize that by Theorem \ref{th1}, since the relevant integral would then be a product of $n$ one dimensional ones, but we omit the statements here. The next example considers general homogeneous phases.

\begin{example}\label{ex2}
If real valued $a\in C^\infty(\mathbb{R}^n\setminus\{0\})$ is homogeneous of degree $m>1$, satisfying (\ref{eq3}) and (\ref{eq4}) globally, then a direct use of Theorem \ref{th1} gives
\begin{equation}\label{eqex}
\left|\partial_x^\alpha\int_{\mathbb{R}^n}e^{i(ta(\xi)+x\cdot\xi)}d\xi\right|\leq C|t|^{-\frac{n+|\alpha|}{m}}(1+|t|^{-\frac1m}|x|)^{-\mu_{|\alpha|}}\quad\text{if}~t\neq0,x\in\mathbb{R}^n,
\end{equation}
for $\alpha\in\mathbb{N}_0^n$, and
\begin{equation*}
\begin{split}
&\left|(1-\Delta)^\frac b2\int_{\mathbb{R}^n}e^{i(ta(\xi)+x\cdot\xi)}d\xi\right|\\
\leq&\begin{cases}C|t|^{-\frac{n+b}{m}}(1+|t|^{-\frac{1}{m}}|x|)^{-\mu_b}&\text{if}~0<|t|<t_0,~x\in\mathbb{R}^n~\text{or}~|t|\geq t_0,~|t|^{-1}|x|>N;\\
C|t|^{-\frac nm}(1+|t|^{-\frac{1}{m}}|x|)^{-\frac{n(m-2)}{2(m-1)}}&\text{if}~|t|\geq t_0,~|t|^{-1}|x|\leq N,
\end{cases}
\end{split}
\end{equation*}
for $b\in[-\frac n2,+\infty)$ and any fixed $t_0$, $N>0$.
\end{example}
It's also interesting to derive (\ref{eqex}) in the following way. Let $\psi(\xi)=\xi^\alpha$ in (\ref{eq2}), by scaling we have
\begin{equation*}
I(t,x)=|t|^{-\frac{n+|\alpha|}{m}}I(\mathrm{sgn}(t),|t|^{-\frac1m}x)=|t|^{-\frac{n+|\alpha|}{m}}(I'(t,x)+I''(t,x)),
\end{equation*}
where
\begin{equation*}
I'(t,x)=\int_{\mathbb{R}^n}e^{i(\mathrm{sgn}(t)a(\xi)+|t|^{-\frac1m}x\cdot\xi)}\xi^\alpha\phi(\xi)d\xi,
\end{equation*}
and $\phi\in C_c^\infty(\mathbb{R}^n)$ equals $1$ near the origin. Applying Paley-Wiener theorem and Lemma \ref{lm1} to $I'$ and $I''$ respectively of course yields (\ref{eqex}); furthermore, if $a(\xi)=|\xi|^m$, due to \cite[Proposition 5.1(ii)]{Mi3}, we also have
\begin{equation*}
|I''(t,x)|\approx(1+|t|^{-\frac1m}|x|)^{-\mu_{|\alpha|}}\quad\mathrm{if}~|t|^{-\frac1m}|x|~\mathrm{is~sufficiently~large},
\end{equation*}
then by the smallness of $I'$ due to Paley-Wiener theorem, we obtain
\begin{equation}\label{eqex'}
|I(t,x)|\approx|t|^{-\frac{n+|\alpha|}{m}}(1+|t|^{-\frac1m}|x|)^{-\mu_{|\alpha|}}\quad\mathrm{if}~|t|^{-\frac1m}|x|~\mathrm{is~sufficiently~large}.
\end{equation}
Notice that through the proof of Theorem \ref{th1} for obtaining (\ref{eqex}), the estimates of $I$ are determined by Lemma \ref{lm1} in domains that Remark \ref{rk0} and Remark \ref{rk00} concern; and by Lemma \ref{lm2} in the domain that Remark \ref{rk3} concerns. These facts and (\ref{eqex'}) explain the meanings of the above remarks through this specific example. It seems that under the spirit of \cite[Proposition 2]{Mi2}, the above argument can also be adapted to general homogeneous $a$ of degree $>1$, as long as the Gaussian curvature of the level surface $\Sigma=\{\xi\in\mathbb{R}^n;~a(\xi)=1\}$ never vanishes.

\subsection{Global smoothing effects for dispersive equations.}
In this part, we shall derive $L^p-L^q$ as well as Strichartz type estimates for the related dispersive equations via Theorem \ref{th1}. Consider the Cauchy problem
\begin{align}\label{eq3.4}
\begin{cases}
\partial_tu=ia(D)u,\\
u(0,\cdot)=u_0\in L^p(\mathbb{R}^n).
\end{cases}
\end{align}
The solution operator $W(t)$ of \eqref{eq3.4} is formally given by
\begin{align}\label{eq3.5}
	u(t,\cdot)=W(t)u_0={\mathcal{F}}^{-1}(e^{ita(\cdot)})*u_0,
\end{align}
where ${\mathcal F}^{-1}$ denotes the inverse Fourier transform which restricting to $C_c^\infty(\mathbb{R}^n)$ reads $f\mapsto(2\pi)^{-n}\int e^{i\xi\cdot x}f(\xi)d\xi$. More generally, we define
\begin{align}\label{eq3.6}
W_b(t)u_0={\mathcal F}^{-1}(e^{ita(\cdot)}\langle\cdot\rangle^b)*u_0=(2\pi)^{-n}I(t,\cdot)*u_0,\quad t\in{\mathbb R},
\end{align}
where $I$ is of the same form as in \eqref{eq05} with $\psi(\xi)=\langle\xi\rangle^b:=(1+|\xi|^2)^{\frac{b}{2}}$ for some real number $b$.

To consider estimates for $W_b(t)$, from now on, we assume that $a(\xi)$ satisfies the conditions in Theorem \ref{th1} with $2\leq m_1\leq m_2$. Assume $b\in [0,\frac{n(m_1-2)}{2}]$, then $\psi(\xi)=\langle\xi\rangle^b$ satisfies the conditions in Theorem \ref{th1} with $b_1=0$ and $b_2=b$. We also denote $\upsilon_j=\frac{n(m_j-2)-2b}{2(m_j-1)}$ for $j=1,~2$, assume $\frac1p+\frac{1}{p'}=1$ in general, and set the quadrangle
\begin{align}\label{eq3.7}
\square_{ABCD}=\{\mbox{$(\frac1p,\frac1q)$}\in\mathbb{R}_+^2;\ \mbox{$({1\over p},{1\over q})$}\ {\rm lies\
		in\ the\ closed\ quadrangle\ ABCD}\}.
\end{align}
Here $A=(\frac{1}{p_0},\frac{1}{p'_0})$, $B=(1,\frac{1}{p'_1})$, $C=(1,0)$ and $D=(\frac{1}{p_1},0)$, where
\begin{equation*}
\begin{split}
p_0=\begin{cases}
2&\mathrm{if}~m_2=2,\\
\frac{2n(m_2-2)}{n(m_2-2)+2b}&\mathrm{if}~m_2>2,
\end{cases}
\end{split}
\quad\text{and}\quad p_1={n\over n-\upsilon_2}.
\end{equation*}
Notice that if $m_1=m_2=2$, i.e., $a(D)$ is a second order pseudo-differential operator, then $b=\upsilon_1=\upsilon_2=0$, and the quadrangle $ABCD$ in Figure \ref{fig1} below collapses to the line segment $OC$. Besides, in $\mathbb{R}^n$ we shall use the standard notations: $L^p$ for the Lebesgue space, $L^{p,q}$ for the Lorentz space, $H^1$ for the Hardy space, and $\text{BMO}$ for the space of functions of bounded mean oscillation. Before stating our result, we set
\begin{equation}\label{equ3.555}
\tau(p,q,\epsilon)=
\begin{cases}
\frac{n}{m_1}(\frac1p-\frac1q)+\frac{b}{m_1}-(1-\frac{p_0'}{p'})(1-\frac{1}{m_1})(\frac{n}{s}-\upsilon_1)&\mathrm{if}~s< \frac{n}{\upsilon_1},\\[4pt]
\frac{n}{m_1}(\frac1p-\frac1q)+\frac{b}{m_1}&\mathrm{if}~s> \frac{n}{\upsilon_1},\\[4pt]
\frac{n}{m_1}(\frac1p-\frac1q)+\frac{b}{m_1}-\epsilon&\mathrm{if}~s=\frac{n}{\upsilon_1},
\end{cases}
\end{equation}
where $s=\frac{q(p'-p_0')}{p'-q}$.
\begin{theorem}{\label{th2}}
In addition to the above assumptions, if $\epsilon>0$ and $(p,q)\in\square_{ABCD}$, we have
\begin{equation}\label{equ3.5}
\|W_b(t)\|_{L^p_*-L^q_*}\leq
\begin{cases}
C|t|^{-\frac{n}{m_2}(\frac1p-\frac1q)-\frac{b}{m_2}}&\text{if}~0<|t|<1,\\[4pt]
C|t|^{-\tau(p,q,\epsilon)}&\text{if}~|t|\geq 1,
\end{cases}
\end{equation}
where
\begin{align}\label{equ3.6}
L^p_*-L^q_*=
\begin{cases}
L^1-L^{p_1',\infty}\, \text{or}\, H^1-L^{p_1'}&\text{if}\,\, (p,q)=(1,p_1'),\\[4pt]
L^{p_1,1}-L^\infty\,  \text{or}\, L^{p_1}-\text{BMO}&\text{if}\,\,(p,q)=(p_1,\infty),\\[4pt]
L^p-L^q&\text{otherwise}.
\end{cases}
\end{align}
\end{theorem}
\begin{figure}
	\centering\includegraphics{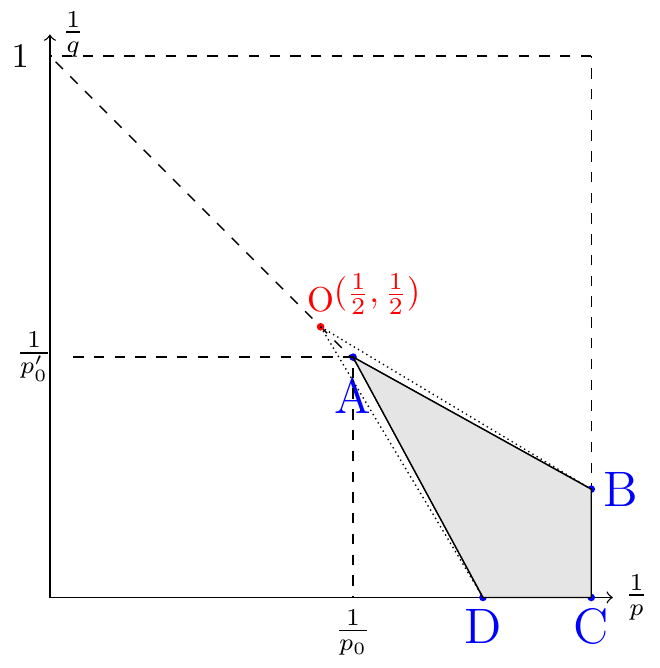}
	\caption{$L^p-L^q$ estimates}\label{fig1}
\end{figure}
\begin{proof}
We divide the proof into the following 2 steps.
	
{\bf Step 1.}  First consider the case that $({1\over p},{1\over q})$ lies in the line segment BC but $(p,q)\ne(1,p_1')$, i.e., $p=1$ and $q\in(p_1',+\infty]$. In this case one checks that $\upsilon_2q>n$. If $0<|t|<1$, by \eqref{eq26} we have
\begin{align*}
\|I(t,\cdot)\|_{L^q}\leq Ct^{-\frac{n+b}{m_2}}\left(\int_{\mathbb{R}^n}{(1+|t|^{-1/{m_2}}|x|)^{-\upsilon_2q}\,dx}\right)^{\frac 1q}\leq C|t|^{-\frac{n}{m_2}(1-\frac1q)-\frac{b}{m_2}}.
\end{align*}
If $|t|\geq1$, also by \eqref{eq26} we have
\begin{align*}
\|I(t,\cdot)\|_{L^q}&\leq Ct^{-\frac{n+b}{m_1}}\left(\int_{|x|< |t|}{(1+|t|^{-1/{m_1}}|x|)^{-\upsilon_1q}\,dx}\right)^{\frac 1q}+Ct^{-\frac{n+b}{m_2}}\left(\int_{|x|\ge |t|}{(1+|t|^{-1/{m_2}}|x|)^{-\upsilon_2q}\,dx}\right)^{\frac 1q}\\
&\leq C\left(|t|^{-\sigma(q, n)}+|t|^{-(\frac{n}{2}-\frac{n}{q})}\right),
\end{align*}
where
\begin{equation}\label{equ3.100}
\sigma(q, n)=
\begin{cases}
\frac{n}{2}-\frac{n}{q}&\text{if}\,\, q< \frac{n}{\upsilon_1},\\[4pt]
\frac{n}{m_1}(1-\frac1q)+\frac{b}{m_1}&\text{if} \,\, q> \frac{n}{\upsilon_1},\\[4pt]
\frac{n}{2}-\frac{n}{q}-\epsilon&\text{if}\,\, q=\frac{n}{\upsilon_1},
\end{cases}
\end{equation}
and $\epsilon>0$. Since $q>\frac{n}{\upsilon_1}$ implies $\frac{n}{m_1}(1-\frac1q)+\frac{b}{m_1}<\frac{n}{2}-\frac{n}{q}$, the Young's inequality gives
\begin{equation}\label{equ3.7}
\|W_b(t)\|_{L^1-L^q}\le C\|I(t,\cdot)\|_{L^q}\leq
\begin{cases}
C|t|^{-\frac{n}{m_2}(1-\frac1q)-\frac{b}{m_2}}&\text{if}\,\,\, 0<|t|< 1,\\[4pt]
C|t|^{-\sigma(q, n)}&\text{if} \,\,\, |t|\ge 1.
\end{cases}
\end{equation}
If $(p,q)=(1,p_1')$, estimate \eqref{equ3.7} with $L^{p_1'}$ replaced by $L^{p_1',\infty}$ (or $L^1$ replaced by $H^1$) follows from the weak Young's inequality (see e.g. \cite[p. 22]{Gro}), Theorem \ref{th1} and the
boundedness of the Riesz potential $f\mapsto|\cdot|^{-n(1-p_1^{-1})}*f$ (see e.g. \cite[p. 3]{Gro2}).
	
{\bf Step 2.} If $(p,q)=(p_0, p_0')$, similar to \cite[Theorem 3.2]{KPG} (also see \cite[Theorem 2]{C1}), let's consider an analytic family of operators $T_z$ with $z=b+iy$ defined by
\begin{equation*}
T_{b+iy}(t)u_0=\mathcal{F}^{-1}\left(\langle\cdot\rangle^{b+iy}e^{ita(\cdot)}\right)\ast u_0,\quad b\in[0,\mbox{$\frac{n(m_1-2)}{2}$}],~y\in \mathbb{R}.
\end{equation*}
Clearly we have $\|T_{iy}\|_{L^2-L^2}\leq C$. On the other hand, the $L^1-L^{\infty}$ estimates for $T_{b+iy}$ where $b=\frac{n(m_1-2)}{2}$ follow from Step 1 above; moreover, from the proof of Theorem \ref{th1}, we see that there is at most a polynomial growth for the parameter $|y|$ in such estimates. Then by the Stein's analytic interpolation theorem (see e.g. \cite[p. 37]{Gro}) we obtain
\begin{equation}\label{equ3.76}
\|W_b(t)\|_{L^{p_0}-L^{p_0'}}\leq
\begin{cases}
C|t|^{-\frac{n}{m_2}(\frac{2}{p_0}-1)-\frac{b}{m_2}}&\text{if}\,\,\, 0<|t|< 1,\\[4pt]
C|t|^{-\frac{n}{m_1}(\frac{2}{p_0}-1)-\frac{b}{m_1}}&\text{if} \,\,\, |t|\ge 1.
\end{cases}
\end{equation}
If $(\frac1p,\frac1q)$ lies in the interior of triangle ABC, we can write $\frac 1p=\frac{\theta}{1}+\frac{1-\theta}{p_0}$ with $\theta=1-\frac{p_0'}{p'}$, and deduce from the Riesz-Thorin interpolation theorem that
\begin{equation*}
\|W_b(t)\|_{L^p-L^q}\leq C\|W_b(t)\|_{L^1-L^s}^{1-\frac{p_0'}{p'}}\|W_b(t)\|_{L^{p_0}-L^{p_0'}}^{\frac{p_0'}{p'}}
\leq
\begin{cases}
C|t|^{-\frac{n}{m_2}(\frac1p-\frac1q)-\frac{b}{m_2}}&\text{if}\,\,\, 0<|t|< 1,\\[4pt]
C|t|^{-\tau(p,q,\epsilon)}&\text{if}\,\, |t|\ge 1,
\end{cases}
\end{equation*}
where $s=\frac{q(p'-p_0')}{p'-q}$, and $s>p_1'$ is obvious by connecting $(\frac1p,\frac1q)$ and A in Figure. On the other hand, when $(\frac1p,\frac1q)$ lies in the edge AB, the above estimate (with $L^s$ replaced by $L^{p_1',\infty}$) follows from the Marcinkiewicz interpolation theorem (see e.g. \cite[p. 31]{Gro}). Therefore estimate \eqref{equ3.5} is valid if $(\frac1p,\frac1q)$ lies in the triangle $ABC$, while the case of triangle $ADC$ follows by duality.
\end{proof}

As an application of Theorem \ref{th2}, we can also establish the closely related Strichartz type estimates. We recall $W(t)$ $(=e^{ita(D)})$ in \eqref{eq3.5}, define $(\mathcal{A}f)(t,\cdot)=\int_{0}^{t}{W(t-\omega)f(\omega,\cdot)d\omega}$ and denote by $L^{p}_{\alpha}=(1-\Delta)^{-\frac{\alpha}{2}}L^p$ the Bessel potential space for some $\alpha>0$.
\begin{corollary}\label{cor3.1}
Under the assumptions before Theorem \ref{th2}, for any fixed $T>0$, we have
\begin{align}
\|W(t)u_0\|_{L^q(\mathbb{R};~L^p_{b/2})}\leq C\|u_0\|_{L^2},\quad&u_0\in L^2\label{eq3.11},\\
\|\mathcal{A}f\|_{L^q((-T,T);~L^p_b)}\leq C\|f\|_{L^{q'}((-T,T);~L^{p'})},\quad&f\in L^{q'}((-T,T);~L^{p'}),\label{equ3.13}
\end{align}
where $p\in[p_0',\frac{2n}{n-m_2+b})$ and $\frac2q=\frac{n}{m_2}(1-\frac2p)+\frac{b}{m_2}$.
\end{corollary}
\begin{proof}
\eqref{eq3.11} follows directly from \eqref{equ3.5}, and we only sketch the proof of \eqref{equ3.13}. In the view of \eqref{equ3.76}, we have
\begin{align}\label{equ3.14}
\|(\mathcal{A}f)(t,\cdot)\|_{L^p_b}\leq& C\int_0^{t}{\|W_b(t-\omega)f(\omega,\cdot)\|_{L^p}d\omega}\nonumber\\
\leq&C\int_0^{t}{|t-\omega|^{-k(t,\omega)}\|f(\omega,\cdot)\|_{L^{p'}}d\omega}
\end{align}
where $k(t,\omega)=\frac{n}{m_2}(\frac{2}{p}-1)+\frac{b}{m_2}$, if $|t-\omega|<1$ and $k(t,\omega)=\frac{n}{m_1}(\frac{2}{p}-1)+\frac{b}{m_1}$, if $|t-\omega|\geq1$. In order to treat the integral \eqref{equ3.14}, we recall the following variant of Hardy-Littlewood-Sobolev inequality (see e.g. \cite[Lemma 2]{GPW}) in dimension one: assume $\eta(t)=|t|^{-\theta_1}$ if $|t|<1$ and $\eta(t)=|t|^{-\theta_2}$ if $|t|\ge1$, where $\theta_1\leq\theta_2$ and $0<\theta_1<1$, then
\begin{align}\label{equ3.15}
\|\eta\ast f\|_{L^{r_2}(\mathbb{R})}\leq C\|f\|_{L^{r_1}(\mathbb{R})},\quad f\in C_c^\infty(\mathbb{R}),
\end{align}
where $1<r_1<r_2<+\infty$, and $\frac{1}{r_1}-\frac{1}{r_2}=1-\theta_1$.

If we set $\theta_1=\frac{n}{m_2}(1-\frac 2p)+\frac{b}{m_2}$ and $\theta_2=\frac{n}{m_1}(1-\frac 2p)+\frac{b}{m_1}$, then $m_1\leq m_2$ implies $\theta_1\leq\theta_2$, and assumption $p<\frac{2n}{n-m_2+b}$ implies $\theta_1<1$. Hence the desired estimate \eqref{equ3.13} follows from \eqref{equ3.14}, \eqref{equ3.15} where $r_1=q'$, and the assumption $\frac2q=\frac{n}{m_2}(1-\frac2p)+\frac{b}{m_2}$.
\end{proof}

We mention that if the phase $a$ does not satisfy the condition \eqref{equ3.1}, we can also establish the same type of estimates as \eqref{equ3.5}. A typical example is that $a$ is an inhomogeneous non-degenerate elliptic polynomial, which has been considered in literatures mentioned in the Introduction. But we would like to consider non-polynomial cases in the following.

\begin{proposition}\label{cor3.2}
Let $a\in C^\infty(\mathbb{R}^n)\cap S^{m}(\Omega)$ be real valued for some $m>1$, where $\Omega=\{\xi\in\mathbb{R}^n;~|\xi|>r_0\}$ for some $r_0>0$. Suppose $a$ satisfies \eqref{eq3} and \eqref{eq4} in $\Omega$, $b\in[0,\frac{n(m-2)}{2}]$, and $(\frac1p,\frac1q)\in\square_{ABCD}$ defined in \eqref{eq3.7} where $m_2=m$, then we have
\begin{align*}
\|W_b(t)\|_{L^p_*-L^q_*}\leq
\begin{cases}
C|t|^{n|\frac 1q+\frac{1}{p}-1|}&\text{if}\,\,\, |t|\geq 1,\\[4pt]
C|t|^{-\frac{n}{m}(\frac1p-\frac1q)-\frac{b}{m}}&\text{if}\,\,\, 0<|t|<1.
\end{cases}
\end{align*}
\end{proposition}

\begin{proof}
It suffices to establish point-wise estimates for $I(t,x)$. We write
$$
\mathcal{F}^{-1}(e^{ita}\langle\cdot\rangle^b)(x)=\mathcal{F}^{-1}(e^{ita}\langle\cdot\rangle^b\phi)(x)+\mathcal{F}^{-1}(e^{ita}\langle\cdot\rangle^b(1-\phi))(x)=I_1+I_2,
$$
where $\phi\in C_c^\infty(\mathbb{R}^n)$ such that $\langle\cdot\rangle^b(1-\phi)\in S^b(\Omega)$. It follows from
Lemma \ref{lm1} that $I_2$ satisfies the estimates \eqref{eq5} and \eqref{eq6}. On the other hand, since $e^{ita}\langle\cdot\rangle^b\psi\in
C_c^\infty(\mathbb{R}^n)$ for every $t\in{\mathbb{R}}$, similar to the standard proof of the Paley-Wiener theorem, we obtain from integration by parts that
$$
|I_1|\le C_k(1+|t|)^k(1+|x|)^{-k}\quad\text{if}~(t,x)\in{\mathbb{R}}\times\mathbb{R}^n~\text{and}~k\in\mathbb{N}_0,
$$
which also implies
$$
|I_1|\le C(1+|t|)^{\mu_b}(1+|x|)^{-\mu_b}\quad{\text{if}}\
(t,x)\in{\mathbb{R}}\times\mathbb{R}^n,
$$
because $\mu_b\geq0$. It is not hard to check from the above facts that
\begin{align*}
|I_1(t,x)|\leq
\begin{cases}
C(1+|t|^{-1}|x|)^{-\mu_b}&\text{if}\,\,\, |t|\geq 1,\\[4pt]
C|t|^{-\frac{n+b}{m}}(1+|t|^{-\frac 1m}|x|)^{-\mu_b}&\text{if}\,\,\, 0<|t|<1,
\end{cases}
\end{align*}
thus
\begin{align*}
|\mathcal {F}^{-1}(e^{ita}\langle\cdot\rangle^b)(x)|\leq
\begin{cases}
C(1+|t|^{-1}|x|)^{-\mu_b}&\text{if}\,\,\, |t|\geq 1,\\[4pt]
C|t|^{-\frac{n+b}{m}}(1+|t|^{-\frac 1m}|x|)^{-\mu_b}&\text{if}\,\,\, 0<|t|<1.
\end{cases}
\end{align*}
Consequently,
\begin{equation*}
\|W_b(t)\|_{L^1-L^q}\leq
\begin{cases}
C|t|^{\frac nq}&\text{if}\,\,\, |t|\ge 1,\\[4pt]
C|t|^{-\frac{n}{m}(1-\frac1q)-\frac{b}{m}}&\text{if}\,\,\, 0<|t|<1.
\end{cases}
\end{equation*}
Now the Proposition follows from the same argument in Theorem \ref{th2}.
\end{proof}

At last, we shall use Theorem \ref{th1} to study the $L^p$ estimates for the Cauchy problem of fractional Schr\"odinger equation
\begin{equation}\label{equ3.20}
	\begin{cases}
		\partial_{t}u =i((-\Delta)^{\alpha}+V(x))u, \quad (t,x)\in \mathbb{R}\times \mathbb{R}^n,\\[4pt]
		u(0,x)=u_{0}(x),
	\end{cases}
\end{equation}
where $(-\Delta)^{\alpha}$ is defined through the Fourier symbol $|\xi|^{2\alpha}$, and $V$ is some complex-valued integrable potential in $\mathbb{R}^n$. We recall that a densely defined linear operator $A$ on a Banach space $X$ is called the generator of a $\beta$-times integrated semigroup for some $\beta\geq0$, if there exist $\omega\geq0$ and an exponentially bounded, strongly continuous family $T(t)$ ($t\ge0$) of bounded linear operators on $X$ such that
$$
(\lambda-A)^{-1}f=\lambda^{\beta}\int_{0}^{\infty}e^{-\lambda t}T(t)fdt,\,\, \,\,\, f\in X,~\lambda>\omega.
$$
\begin{corollary}\label{cor3.6}
Suppose $(-\Delta)^{\alpha}$ and $V\in L^r$ both have maximal domains in $L^p$, for some $\alpha>1$, $1<p\leq2$ and $r\in(\frac{n}{2\alpha},+\infty]\cap R_{p,\alpha}$ where
\begin{equation*}
R_{p,\alpha}=\begin{cases}\left[p,\frac{(2\alpha-1)p}{\alpha(2-p)}\right)&\text{if}\,\,~1<p<\frac{2\alpha-1}{\alpha},\\
\left(\frac{p}{\alpha(2-p)},\frac{(2\alpha-1)p}{\alpha(2-p)}\right)&\text{if}\,\,~\frac{2\alpha-1}{\alpha}\leq p<2,\\
\{+\infty\} &\text{if}\,\,~p=2.
\end{cases}
\end{equation*}
Denote $H=(-\Delta)^{\alpha}+V$ in the sense of operator summation. Then either $iH$ or $-iH$ is a generator of a $\beta$-times integrated semigroup on $L^p$ if $\beta>n|\frac12-\frac1p|+\frac1p$. Further, there exist constants $C,~\omega>0$, such that for every initial data $u_0\in \text{Dom}(\omega+iH)^{\beta}\cap\text{Dom}(\omega-iH)^{\beta}$, the Cauchy problem \eqref{equ3.20} has a unique solution $u\in C(\mathbb{R};~L^p(\mathbb{R}^n))$ satisfying
\begin{align*}
	\|u(t,x)\|_{L^p}\leq\begin{cases}Ce^{\omega |t|}\|(\omega+iH)^{\beta}u_0\|_{L^p}&\text{if}~t\geq0,\\
		Ce^{\omega |t|}\|(\omega-iH)^{\beta}u_0\|_{L^p}&\text{if}~t<0.
	\end{cases}
\end{align*}
\end{corollary}

\begin{proof}
For the first statement we only prove it with $iH$, and the other case follows by the same argument. We shall need the following result due to Kaiser and Weis \cite[Theorem 3.3 (a)]{KW}: assume that $(A,\text{Dom}(A))$ is the generator of a $\beta$-times integrated semigroup on $L^p$ $(1<p\leq2)$, $(B,\text{Dom}(B))$ is a linear operator on $L^p$ such that $\text{Dom}(A)\subset \text{Dom}(B)$, and there are constants $M,\omega_0\ge 0$ such that
\begin{align}\label{equ3.17}
\|B(\lambda-A)^{-1}\|_{L^p-L^p}\leq M<1\quad\text{if}\,\, \Re \lambda>\omega_0,
\end{align}
then $(A+B, \text{Dom}(A))$ is the generator of a $\beta$-times integrated semigroup on $L^p$ for any $\beta>n|\frac12-\frac1p|+\frac1p$.

Note that $i(-\Delta)^{\alpha}$ generates a $\beta$-times integrated semigroup on $L^p$ if $\beta>n|\frac12-\frac1p|$ (see \cite[Theorem 4.2]{Hieber}). Hence it suffices to show that $\text{Dom}(A)\subset \text{Dom}(B)$ and \eqref{equ3.17} are true for $A=i(-\Delta)^{\alpha}$ and $B=iV$. We first observe that
\begin{align*}
(\lambda-i(-\Delta)^{\alpha})^{-1}f=\int_{0}^{\infty}e^{-\lambda t}e^{it(-\Delta)^{\alpha}}fdt,\quad \text{if}~\Re\lambda>0,~f\in C_c^{\infty}(\mathbb{R}^n).
\end{align*}
Applying Theorem \ref{th2} with $m_1=m_2=2\alpha$ and $b=0$, we obtain if $(\frac1p,\frac1q)$ lies in the interior of the quadrangle $\Box_{OBCD}$ in figure \ref{fig1}, that $s=\frac{q(p'-p_0')}{p'-q}>\frac{n}{\upsilon_1}$, and therefore
$$
\|e^{it(-\Delta)^{\alpha}}\|_{L^p-L^q}\leq C|t|^{-\frac{n}{2\alpha}(\frac1p-\frac1q)},\quad\text{if}\,\, t\ne 0.
$$
Since $r\in R_{p,\alpha}$, we can choose the above $p$ and $q$ such that $\frac1p-\frac1q=\frac1r$, and have $\frac{n}{2\alpha}(\frac1p-\frac1q)=\frac{n}{2\alpha}\cdot\frac 1r<1$ by the assumption $r>\frac{n}{2\alpha}$, which implies
\begin{align*}
\|(\lambda-i(-\Delta)^{\alpha})^{-1}\|_{L^p-L^q}&\leq \int_{0}^{\infty}e^{-\Re\lambda t}t^{-\frac{n}{2\alpha}(\frac1p-\frac1q)} dt\nonumber=C|\Re\lambda|^{\frac{n}{2\alpha}(\frac1p-\frac1q)-1}.
\end{align*}
Then by H\"{o}lder's inequality, we have $\text{Dom}(i(-\Delta)^{\alpha})\subset\text{Dom}(iV)$, and there exists $\omega_0\ge1$ such that
\begin{align*}
\|iV(\lambda-i(-\Delta)^{\alpha})^{-1}\|_{L^p-L^p}&\leq \|V\|_{L^q-L^p}\|(\lambda-i(-\Delta)^{\alpha})^{-1}\|_{L^p-L^q}\\
&\leq \|V\|_{L^r}|\Re\lambda|^{\frac{n}{2\alpha}(\frac1p-\frac1q)-1}\\
&\leq \mbox{$\frac12$}\quad\quad \text{if}~|\Re\lambda|\ge \omega_0,
\end{align*}
where the last inequality comes form the assumption $r>\frac{n}{2\alpha}$ again. The second statement then follows by the first statement and van Neerven and Straub \cite[Theorem 1.1]{VS}, which completes the proof.
\end{proof}

\begin{remark}\label{rk3.1}
We note that the above Corollary focuses on concerning complex-valued potentials, while there are better results for real-valued potentials, whose proofs are totally different. In the free case, i.e., $V=0$, studies on Fourier multipliers give (see e.g. \cite[Theorem 4.1]{Mi3}) for every fixed $\omega\in\mathbb{C}\setminus(-\infty,0]$ that
\begin{equation}\label{equ3.21}
\|e^{it(-\Delta)^{\alpha}}(\omega+(-\Delta)^{\alpha})^{- \beta}\|_{L^p-L^p}\leq C(1+|t|)^{n|\frac12-\frac1p|},\quad t\in\mathbb{R},
\end{equation}
if $\alpha>0$, $\beta\geq n|\frac12-\frac1p|$ when $1<p<+\infty$, and $\beta>\frac n2$ when $p=1$. Notice that the range of $\beta$ is sharp if $\alpha\neq\frac12$. For some real-valued $V$, or even more generally, a self-adjoint operator $H$ in $L^2$, if $\{e^{-zH}\}_{\Re z>0}$ is a bounded analytic semigroup on $L^2$ whose kernel satisfies for some $m\in\mathbb{N}_+$ the Gaussian type estimate
\begin{equation*}
|K(t,x,y)|\leq c_{1}t^{-\frac{n}{2m}}\exp\left\{-c_{2}\frac{|x-y|^{\frac{2m}{2m-1}}}{t^{\frac{1}{2m-1}}}\right\},\quad\text{if}~ t>0,~x,~y\in\mathbb{R}^n,
\end{equation*}
then Carron, Coulhon and Ouhabaz \cite[Theorem 5.3]{CCO} shows that $\pm iH$ with maximal domain generates a $\beta$-times integrated semigroup on $L^p$ if $\beta>n|\frac12-\frac1p|$, and a slightly weaker estimate than \eqref{equ3.21} with $(-\Delta)^\alpha$ replaced by $\pm H$ holds. We also note that recently, Ancona and Nicona \cite{DN} can weaken the assumption on point-wise heat kernel estimate but obtain sharp $L^p$ estimate in the form of \eqref{equ3.21} with $(-\Delta)^\alpha$ replaced by $\pm H$.
\end{remark}

\noindent
{\bf Acknowledgements:} The first author was visiting The University of Chicago while this research was partly carried out, supported by the China Scholarship Council. The third author was supported by the National Natural Science Foundation of China (No. 11471129).


\end{document}